\documentclass{amsart}

\usepackage{amssymb,color}
\usepackage{amsfonts}
\usepackage{amsmath}
\usepackage{euscript}
\usepackage{enumerate}
\usepackage{pdfsync}
\newtheorem{theorem}{Theorem}[section]

\newtheorem{prop}[theorem]{Proposition}

\newtheorem*{Theorem1'}{Theorem 1'}

\theoremstyle{definition}

\theoremstyle{remark}

\newcommand \GL{{\mathrm{GL}}}
\newcommand \Z{{\mathbb Z}}

\newcommand \N{{\mathbb N}}

\def\a{\alpha}
\def\Hol{{\mathrm{Hol}}}
\def\b{\beta}

\begin{document}

\title[Sums of binomial coefficients modulo $p$ and  
groups of exponent $p^n$]{Sums of binomial coefficients in Pascal's triangle taken modulo $p$ and their relation to 
groups of exponent $p^n$}

\author{Fernando Szechtman}
\address{Department of Mathematics and Statistics, University of Regina, Canada}
\email{fernando.szechtman@gmail.com}
\thanks{The author was partially supported by NSERC grant 2020-04062}



\begin{abstract} We give a simple matrix-based proof of congruence equations modulo a prime~$p$
involving sums of binomial coefficients appearing in Pascal's triangle. These equations can be used to construct
some groups of exponent $p^n$. These groups, as well as others of exponent $p^{n+1}$, explain why $p=2$ is not really
an exceptional prime in relation to the Heisenberg group over the field with $p$ elements.
\end{abstract}

\maketitle

\section{Introduction}

It is common for the prime 2 to play the role of villain in Mathematics. We consider one such a case
and then we use the sum of binomial coefficients appearing in Pascal's triangle taken
modulo a prime to show that, in fact, all primes behave in the exactly the same manner.

Some notation is required for this purpose. We fix throughout $p,m\in\N$, with $p$ a prime, and write $F$ for the field with $p$ elements. 
We let $J\in M_m(F)$ and $A=I+J\in M_m(F)$
stand for the upper triangular Jordan blocks with eigenvalues 0 and 1, respectively.
We also fix an $F$-vector space $V$ of dimension $m$. For $v\in V$ and 
$\gamma\in\mathrm{End}(V)$, we write $v^\gamma=v\gamma$ for the result of applying $\gamma$ to $v$. In keeping with this
notation function composition is considered from left to right, and the matrix of an endomorphism of $V$ relative to a basis of $V$ is constructed row by row instead of column by column. Let $G$ be a group. If $x,y\in G$, we write $x^y=y^{-1}xy$. If $G$ is  finite,
its exponent
is the smallest natural number $m$ such that $x^m=1$ for all $x\in G$. Given a subgroup $H$ of $\mathrm{Aut}(G)$, we  write $\Hol(G,H)$
for the holomorph of $G$ relative to $H$. This is the group having (copies of) $G$ and~$H$ as subgroups, with $G$ normal,
$G\cap H$ trivial, $\Hol(G,H)=HG$, and if $g\in G$ and $\a\in H$, then
$\a^{-1}g\a=g^\a$ is the result of applying $\a$ to $g$. If $H=\langle\a\rangle$, we write $\Hol(G,\a)$ instead of $\Hol(G,H)$.

Consider the Heisenberg
group $H(p)$, namely the group of all $3\times 3$ upper triangular matrices with 1's along the main diagonal.
Then $H(p)\cong\Hol(V,\a)$, where $m=2$ and $\a\in\GL(V)$ can be represented by $A$ relative to some basis of $V$.
It is well known and easy to see that the exponent of $H(p)$ is $p$ when $p$ is odd and 4 when $p=2$.
The reason for this dichotomy is the fact that $1+\cdots+p-1$ is congruent to 0 or 1 modulo $p$ depending on whether $p$ is odd
or $p=2$, respectively. There seems to be nothing that can salvage the prime 2 from its unique role in this setting.
However, a wider view of Pascal's triangle reveals 
that all primes are equally exceptional.

Perhaps the two most well-known results concerning binomial coefficients and their divisibility by a given prime are the following theorems
of Kummer \cite{K} and Lucas \cite{L}.

\medskip

\noindent{\bf Theorem} (Kummer). Let $p$ be a prime. Given integers $0\leq m\leq n$ the number of times that $p$ divides
the binomial coefficient $\binom{n}{m}$ is the amount of carries when adding $m$ and $n-m$ in base $p$.

\medskip

\noindent{\bf Theorem} (Lucas). Let $p$ be a prime. Given integers $0\leq m\leq n$, such that $m=(m_\ell\dots m_1m_0)_p$
and $m=(n_\ell\dots n_1n_0)_p$, we have
$$
\binom{n}{m}\equiv \underset{0\leq i\leq\ell} \Pi \binom{n_i}{m_i}\mod p,
$$
where $\binom{0}{0}=1$ and $\binom{a}{b}=0$ if $a<b$.

\medskip

%We refer to \cite{M} and \cite{R} and references therein for recent results on this topic.
In these note we are concerned with {\em sums} of binomial coefficients when taken modulo a prime, with an application to
group theory.

\section{Congruence equations mod $p$ in Pascal's triangle and groups of exponent $p^n$}

\begin{prop}\label{min} Suppose that $n\in\N$ satisfies $m<p^n$. Then 
\begin{equation}\label{poly0}
I+A+\cdots+A^{p^n-1}=0.
\end{equation}
\end{prop}

\begin{proof} The minimal polynomial of $A$ is $(X-1)^m$, which is a factor of 
\begin{equation}\label{poly}
1+X+\cdots+X^{p^n-1}=(X^{p^n}-1)/(X-1)=(X-1)^{p^n}/(X-1)=(X-1)^{p^n-1}.
\end{equation}
\end{proof}

\begin{prop} Suppose that $n\in\N$ and $i\in\Z$ satisfy $0\leq i<p^n-1$, and set
$$
S_{p^n,i}=\binom{i}{i}+\binom{i+1}{i}+\cdots+\binom{p^n-1}{i}.
$$
Then 
\begin{equation}\label{rj}
S_{p^n,i}\equiv 0\mod p.
\end{equation}
\end{prop}

\begin{proof} Take $m=p^n-1$. Then $I+A+\cdots+A^{p^n-1}=0$ by Proposition \ref{min}. On the other hand
\begin{equation}\label{triangle2}
I+A+\cdots+A^{p^n-1}=S_{p^n,0}\cdot I+ S_{p^n,1}\cdot J+\cdots+S_{p^n,p^n-2}\cdot J^{p^n-2}+S_{p^n,p^n-1}\cdot J^{p^n-1}.
\end{equation}
For $0\leq i<p^n-1$, $J^i$ is the upper triangular matrix with 1's along the $i$th superdiagonal and 0's everywhere else.
Thus, $I,J,\dots,J^{p^n-2}$ are linearly independent over $F$. Since $J^{p^n-1}=0$, it follows that $S_{p^n,i}\equiv 0\mod p$ for all 
$0\leq i<p^n-1$.
\end{proof}

Note that if $n\in\N$ and $m=p^n-1$, then making use of (\ref{triangle2}), the linear independence of $I,J,\dots,J^{p^n-2}$, and 
$J^{p^n-1}=0$, the equations (\ref{poly0}) and (\ref{rj}) become equivalent to each other.

\begin{prop}\label{h1} Suppose that $n\in\N$ satisfies $p^{n-1}<m\leq p^n$. Let $\a\in\GL(V)$ be an automorphism of $V$ that is represented by $A$ with respect
to some basis of $V$, and set $G=\Hol(V,\a)$. Then $G$ has order $p^{m+n}$ and exponent $p^n$ if $m<p^n$ and $p^{n+1}$ if $m=p^n$. 
\end{prop}

\begin{proof} The minimal polynomial of $A$ is $(X-1)^m$. This is not a factor of $(X-1)^{p^{n-1}}=X^{p^{n-1}}-1$,
but it is a factor of $(X-1)^{p^{n}}=X^{p^{n}}-1$, so $\a^{p^{n-1}}\neq 1$ but $\a^{p^{n}}=1$. Thus the order
of $\a$ is $p^n$. As $|V|=p^m$, it follows that $|G|=p^{m+n}$.

Let $v\in V$ and $\b\in\langle \a\rangle$. Then $\b=\a^i$ for some $i\in\N$, so that 
$i=p^s j$, where $j\in\N$, $s\geq 0$, and $\gcd(p,j)=1$. Clearly, $\b^{p^{n}}=1$. Set $\gamma=\a^j$. As
$\gcd(p,j)=1$,
then $\a$ is similar to $\gamma$, so $\gamma$ can also be represented by $A$ relative 
to some other basis, say $B=\{v_1,\dots,v_m\}$, of $V$. Set $\delta=1+\gamma+\cdots+\gamma^{p^{n-1}}$ and $\epsilon=\delta^{p^s}$,
so that $1+\b+\cdots+\b^{p^{n-1}}=\epsilon$ and therefore
\begin{equation}\label{delta}
(\b v)^{p^n}=\b^{p^n} v^{\epsilon}=v^{\epsilon}.
\end{equation}

Suppose first that $m<p^n$. Then $\delta=0$ by Proposition \ref{min}, so $(\b v)^{p^n}=1$ by (\ref{delta}), whence
the exponent of $G$ is exactly $p^{n}$. Suppose next that $m=p^n$. Then by (\ref{delta}),
$$
(\b v)^{p^{n+1}}=(v^{\epsilon})^p=1.
$$
Let $E\in M_m(F)$ be the matrix having a 1 in position $(1,m)$ and 0's elsewhere.
Then $(A-1)^{p^n-1}=J^{p^n-1}=E$, so (\ref{poly}) ensures that $\delta$ is represented by $E$ relative to  $B$,
whence
$$
(\gamma v_1)^{p^{n}}=\gamma^{p^{n}} v_1^{1+\gamma+\cdots+\gamma^{p^{n-1}}}=v_1^{\delta}=v_m.
$$
Therefore the exponent of $G$ is precisely $p^{n+1}$.
\end{proof}

Observe that when $m<p^n$ the fact that the exponent of $G$ is $p^n$ is equivalent to (\ref{poly0}) and (\ref{rj}).
Notice also if $n=1$ and $m=2$, then any odd prime $p$ satisfies $p^{n-1}<m<p^n$ and $G\cong H(p)$.
Note as well that if $n=1$ and $m=2$, then $p=2$ satisfies $m=p^n$ and $G\cong H(p)$.

\end{document}